\documentclass[leqno]{article}
\usepackage{amsmath,amssymb,amsthm,amscd,dsfont}
\numberwithin{equation}{section} \allowdisplaybreaks
\begin{document}
\newtheorem{theorem}{Theorem}[section]
\newtheorem{defin}{Definition}[section]
\newtheorem{prop}{Proposition}[section]
\newtheorem{corol}{Corollary}[section]
\newtheorem{lemma}{Lemma}[section]
\newtheorem{rem}{Remark}[section]
\newtheorem{example}{Example}[section]
\title{Dirac structures on generalized Riemannian manifolds}
\author{{\small by}\vspace{2mm}\\Izu Vaisman}
\date{}
\maketitle
{\def\thefootnote{*}\footnotetext[1]%
{{\it 2000 Mathematics Subject Classification: 53C15, 53D99} .
}}
\begin{center} \begin{minipage}{12cm}
A{\footnotesize BSTRACT. We characterize the Dirac structures that are parallel
with respect to Gualtieri's canonical connection of a generalized Riemannian
metric. On the other hand, we discuss Dirac structures that are images of
generalized tangent structures.
These structures turn out to be Dirac structures that, if seen
as Lie algebroids, have a symplectic structure.
Particularly, if compatibility with a generalized Riemannian metric is
required, the symplectic structure is of the K\"ahler type.}
\end{minipage}
\end{center} \vspace{5mm}
\section{Introduction}
The concept of generalized geometry is due to Hitchin
\cite{Ht1} and is of interest in the physical theory of supersymmetry
(e.g., \cite{Zab}). In generalized geometry the tangent bundle $TM$ of the
$m$-dimensional, differentiable manifold $M$ is replaced by the big tangent
bundle $\mathbf{T}M=TM\oplus T^*M$. The latter has the non degenerate, neutral
metric\footnote{In many papers on generalized geometry $g$ is defined by
$(1/2)(\alpha(Y)+\mu(X))$.}
$$ g((X,\alpha),(Y,\mu))=
\alpha(Y)+\mu(X)
$$ and the Courant bracket
$$
[(X,\alpha),(Y,\mu)]=([X,Y],L_X\mu-L_Y\alpha
+\frac{1}{2}d(\alpha(Y)-\mu(X)),$$ where $X,Y\in
\chi^1(N),\,\alpha,\mu\in\Omega^1(N)$ ($\chi^k(M)$ is
the space of $k$-vector fields and $\Omega^k(M)$ is the space of differential
$k$-forms on $M$; we will also use calligraphic characters for pairs:
$\mathcal{X}=(X,\alpha),\mathcal{Y}=(Y,\mu)$, etc.). Thus, the structure group
of $\mathbf{T}M$ is $O(m,m)$ and, by definition, the generalized geometric
structures are reductions of this structure group to various subgroups.

The almost Dirac structures, which are maximal $g$-isotropic subbundles $E$ of
$\mathbf{T}M$ and are important in mechanics and physics
\cite{C}, are generalized structures where $O(m,m)$ is
reduced to the subgroup that preserves a maximal isotropic subbundle of
$\mathds{R}^{2m}$ endowed with the standard neutral metric. The structure $E$ is
a Dirac structure if it is integrable, i.e., closed under the Courant bracket.

Hitchin's work and the subsequent thesis of Gualtieri \cite{Galt} started a
stream of research and publications on generalized complex structures. A
generalized complex structure is a $g$-skew-symmetric endomorphism
$\mathcal{J}\in End(\mathbf{T}M)$ with $\mathcal{J}^2=-Id$ and a vanishing
Courant-Nijenhuis torsion (see Section 3). Generalized complex structures may be
combined with generalized Riemannian structures, which are reductions of the
structure group of $\mathbf{T}M$ to $O(n)\times O(n)$, thus leading to
generalized K\"ahler manifolds \cite{Galt}. In \cite{G2007}, it was proven that
the generalized Riemannian metric produces a canonical connection $\nabla$ on
$\mathbf{T}M$ and the generalized K\"ahler structures are characterized by
$\nabla\mathcal{J}=0$ plus a certain torsion condition. The theory of
generalized complex structures also motivated some work on related generalized
structures: paracomplex, contact, $F$, $CRF$, Sasakian, etc.,
\cite{{PW},{IV},{V-CRF},{VgenS},{Wade}}.

In the present paper, we will discuss the relationship between a Dirac structure
and a generalized Riemannian metric. In Section 2 we give a straightforward
definition of the canonical connection of a generalized Riemannian metric and
compute its torsion. In Section 3, passing through a discussion of generalized
para-Hermitian structures, we show that, on a generalized Riemannian manifold, a
Dirac structure $E$ may be represented by a tensor field $F_E\in Iso (TM)$ and
we get the conditions for $E$ to be preserved by the canonical connection. In
Section 4, we study Dirac structures $E$ that are images of a generalized
tangent structure and show that these are characterized by the existence of a
symplectic structure on the Lie algebroid $E$ with the Courant bracket. In
Section 5, we show that, on a generalized Riemannian manifold, a symplectic form
on the Dirac structure $E$ is equivalent with a K\"ahler type form.

The paper is the initial presentation of some nice generalized, geometric
structures, which further studies will show to be of interest, hopefully.
\section{Generalized Riemannian manifolds}
A {\it generalized, Riemannian structure} on $M$ is a reduction of the structure
group of $(\mathbf{T}M,g)$ from $O(m,m)$ to $O(m)\times O(m)$, i.e., a
decomposition
\begin{equation}\label{descompG}\mathbf{T}M=V_+\oplus V_-\end{equation}
where $V_\pm$ are maximal positive, respectively negative, subbundles of $g$.
Obviously, $rank\,V_\pm=m$ and $V_+\perp_g V_-$, which shows that, in fact, the
reduction is defined by one of these subbundles.

Equivalently \cite{Galt}, the structure may be seen as a positive definite
metric $G$ together with a $G$-orthogonal decomposition (\ref{descompG}) such
that $G|_{V_\pm}=\pm g$. We may define $G$ by the endomorphism $\phi$ of
$\mathbf{T}M$ given by $\phi|_{V_\pm}=\pm Id$, equivalently, by
\begin{equation}\label{eqGg}  G((X,\alpha),(Y,\mu)) =
g(\phi(X,\alpha),(Y,\mu)).
\end{equation}
The endomorphisms $\phi$ that produce generalized Riemannian metrics are
characterized by the conditions
\begin{equation}\label{Riemannbig2}
\phi^2=Id,\;\;g(\phi(X,\alpha),\phi(Y,\mu)) =
g((X,\alpha),(Y,\mu))
\end{equation} and the requirement that $G$ given by (\ref{eqGg})
is positive definite (the second condition (\ref{Riemannbig2}) comes from the
symmetry of $G$ and ensures that the $\pm1$-eigebundles $V_\pm$ of $\phi$ are
$G$-orthogonal).

In \cite{Galt}, it was shown that $G$ is equivalent with a pair $(\gamma,\psi)$,
where $\gamma$ is a usual Riemannian metric on $M$ and $\psi\in\Omega^2(M)$.
This equivalence is realized by putting
\begin{equation}\label{exprEpm}
V_{\pm}=\{(X,\flat_{\psi\pm\gamma}X)\,/\,X\in TM\}.\end{equation} Formula
(\ref{exprEpm}) also shows the existence of isomorphisms
\begin{equation}\label{isotau}\tau_\pm:V_\pm\rightarrow TM,\hspace{3mm}
\tau_\pm(X,\flat_{\psi\pm\gamma}X)=X,\end{equation}
which may be used to transfer structures between $V_\pm$ and $TM$. In
particular, the two metrics $G|_{V_\pm}$ transfer to $\gamma$.

On $\mathbf{T}M$, it is natural to consider connections $\nabla$ that are
compatible with the neutral metric $g$, i.e., such that
\begin{equation}\label{conexg}
X(g(\mathcal{Y},\mathcal{Z}))= g(\nabla_X \mathcal{Y},\mathcal{Z})
+g(\mathcal{Y},\nabla_X
\mathcal{Z}),\hspace{2mm}\forall{X}\in TM,\mathcal{Y},\mathcal{Z}
\in\Gamma \mathbf{T}M;\end{equation} we call
them {\it big connections}.

Furthermore, on a generalized Riemannian manifold $(M,G)$, a connection $\nabla$
on $\mathbf{T}M$ is a $G$-{\it metric} connection if it is compatible with $G$,
i.e., (\ref{conexg}) with $g$ replaced by $G$ holds. If $\nabla$ is a big
connection, condition (\ref{conexg}) for $G$ is equivalent with
$$
\nabla_\mathcal{X}(\phi
\mathcal{Y})-\phi(\nabla_\mathcal{X}\mathcal{Y})=0,$$
which, furthermore, is equivalent with the commutation of $\nabla$ with the two
projections $(1/2)(Id\pm\phi)$. Hence, $\nabla$ is $G$-metric iff it preserves
the subbundles $V_\pm$. By using the transport to $TM$ via $\tau_\pm$, we see
that there exists a bijective correspondence between $G$-metric big connections
$\nabla$ and pairs $D^\pm$ of $\gamma$-metric connections on $M$, which is
realized by
$$ \nabla_X(Y,\flat_{\psi\pm\gamma}Y)
=(D^\pm_XY,\flat_{\psi\pm\gamma}D^\pm_XY).
$$

For any Riemannian metric $\gamma$, there exists a unique $\gamma$-metric
connection with a prescribed torsion. Particularly, a generalized Riemannian
metric $G\Leftrightarrow(\gamma,\psi)$ produces two connections on $TM$, which
are the $\gamma$-metric connections $D^\pm$ with the torsion defined by
\begin{equation}\label{Tdpsi} \gamma(T^{D^\pm}(X,Y),Z)=\pm d\psi(X,Y,Z).
\end{equation}
The connections $D^\pm$ are given by
\begin{equation}\label{DpmcuLC}D^\pm_XY=D_XY\pm
\frac{1}{2}\sharp_\gamma[i(Y)i(X)d\psi],
\end{equation} where $D$ is the Levi-Civita connection of $\gamma$.

The $G$-metric big connection $\nabla$ defined by the connections
(\ref{DpmcuLC}) is called the {\it canonical big connection} of $G$; one can see
that $\nabla$ coincides with the connection defined by Gualtieri
\cite{G2007} and Ellwood \cite{E}.
If we define the {\it Courant torsion}
$$T^{\nabla}(\mathcal{X},\mathcal{Y})=
\nabla_{X}\mathcal{Y}-
\nabla_{Y}\mathcal{X}-
[\mathcal{X},\mathcal{Y}],$$ we get an object that is not
$C^\infty(M)$-bilinear. This is corrected in the {\it Gualtieri torsion}
\cite{G2007}
\begin{equation}\label{Gualttord}
\mathcal{T}^{\nabla}(\mathcal{X},\mathcal{Y},\mathcal{Z})=
g(T^{\nabla}(\mathcal{X},\mathcal{Y}),\mathcal{Z}) +\frac{1}{2}[g(\nabla_{
\mathcal{Z}}\mathcal{X},\mathcal{Y}) -g(\nabla_{
\mathcal{Z}}\mathcal{Y},\mathcal{X})].\end{equation}
The tensorial character and, also, the total skew symmetry of
$\mathcal{T}^\nabla$ follow from the properties of the Courant bracket
\cite{LWX}.

We compute the Gualtieri torsion of the canonical big connection; the results
will agree with those of
\cite{G2007}). For any $X,Y\in\chi^1(M)$, computations give
\begin{equation}\label{Cptpsi} [(X,\flat_\psi X),(Y,\flat_\psi Y)]
=([X,Y],\flat_\psi[X,Y]+i(Y)i(X)d\psi),\end{equation}
\begin{equation}\label{auxgeneral} [(X,\flat_{\psi\pm\gamma}X),
(Y,\flat_{\psi\pm\gamma}Y)] = ([X,Y],\flat_{\psi\pm\gamma}[X,Y]
\end{equation} $$+i(Y)i(X)d\psi
\pm (L_Xi(Y)\gamma-i(X)L_Y\gamma)),$$
\begin{equation}\label{CptPpm} [(X,\flat_{\psi+\gamma}X),(Y,\flat_{\psi
-\gamma}Y)]=([X,Y],\flat_\psi[X,Y])+(0,i(Y)i(X)d\psi\end{equation}
$$-L_X(\flat_\gamma Y)-L_Y(\flat_\gamma X)+d(\gamma(X,Y)).$$

Now, insert (\ref{CptPpm}) in the expression of the {\it mixed Courant torsion}
$$T^\nabla((X,\flat_{\psi+\gamma}X),(Y,\flat_{\psi-\gamma}Y))=
(D^-_XY,\flat_{\psi-\gamma}D^-_XY)- (D^+_YX,\flat_{\psi+\gamma}D^+_YX)$$
$$-[(X,\flat_{\psi+\gamma}X),(Y,\flat_{\psi -\gamma}Y)]$$ where
$D^\pm$ are given by (\ref{DpmcuLC}). After some technical calculations, we
shall obtain
$$T^\nabla((X,\flat_{\psi+\gamma}X),(Y,\flat_{\psi-\gamma}Y))=0.$$
As a matter of fact, the previous equality is equivalent to (\ref{Tdpsi}). The
annulation of the mixed Courant torsion implies
$$
\mathcal{T}^\nabla(\mathcal{X}_+,\mathcal{Y}_+,\mathcal{Z}_-)=0,\;
\mathcal{T}^\nabla(\mathcal{X}_-,\mathcal{Y}_-,\mathcal{Z}_+)=0
\hspace{3mm} \forall \mathcal{X}_\pm,\mathcal{Y}_\pm,
\mathcal{Z}_\pm\in V_\pm.$$

Furthermore, we have
\begin{equation}\label{Ctorscan} pr_{TM}T^\nabla((X,\flat_{\psi\pm\gamma}X),
(Y,\flat_{\psi\pm\gamma}Y)) =T^{D^\pm}(X,Y)\end{equation} and then, with
(\ref{auxgeneral}),
\begin{equation}\label{Ctorscan1}
pr_{T^*M}T^\nabla((X,\flat_{\psi\pm\gamma}X),(Y,\flat_{\psi\pm\gamma}Y))
=\flat_{\psi\pm\gamma}T^{D^\pm}(X,Y)\end{equation}
$$-[i(Y)i(X)d\psi
\pm (L_Xi(Y)\gamma-i(X)L_Y\gamma)].$$
From (\ref{Ctorscan}), (\ref{Ctorscan1}) and (\ref{Tdpsi}), we get
$$ T^\nabla((X,\flat_{\psi\pm\gamma}X),
(Y,\flat_{\psi\pm\gamma}Y)) =(\pm
i(Y)i(X)d\psi,\pm\flat_{\psi\pm\gamma}i(Y)i(X)d\psi
$$
$$-[i(Y)i(X)d\psi\pm (L_Xi(Y)\gamma-i(X)L_Y\gamma)]).$$
If this expression is inserted in (\ref{Gualttord}) and the required technical
computations are performed, the result is
\begin{equation}\label{Gtorsnabla2}
\mathcal{T}^\nabla((X,\flat_{\psi\pm\gamma} X),
(Y,\flat_{\psi\pm\gamma} Y), (Z,\flat_{\psi\pm\gamma} Z))=
2d\psi(X,Y,Z).\end{equation}
\begin{rem}\label{obscurb} {\rm It is easy to see that
the curvature of the canonical big connection is equivalent with the pair of
curvature tensors $R^{D^\pm}$ of the connections $D^\pm$. }\end{rem}

An alternative notion, which we will not use in this paper, is that of a
generalized connection
\cite{G2007}. Assume that the pair $(A,A^*)$ is a Lie bialgebroid
\cite{MX} and $V$ is a vector bundle on $M$. Consider a pair
$(\nabla,\nabla^*)$ where $\nabla$, $\nabla^*$ are an $A$-connection,
respectively an $A^*$-connection on $V$. The operator
$$ \mathcal{D}_{(a,a^*)}v
=(\nabla_av+\nabla^*_{a^*}v)$$ where $a\in\Gamma A, a^*\in\Gamma
A^*$ is called an $(A,A^*)$-{\it generalized connection} or {\it covariant
derivative}. $\mathcal{D}$ is $\mathds{R}$-bilinear and has the properties
$$
\begin{array}{c}\mathcal{D}_{(fa,fa^*)}v
=f\mathcal{D}_{(a,a^*)}v,\vspace{2mm}\\
\mathcal{D}_{(a,a^*)}(fv) =f\mathcal{D}_{(a,a^*)}v
+(\sharp_Aa+\sharp_{A^*}a^*)(f)v\end{array}$$ where
$\sharp_A,\sharp_{A^*}$ are the anchors of $A,A^*$. $\mathcal{D}$ is said to
preserve $g\in\Gamma\otimes^2V^*$ if
$$ (\sharp_Aa+\sharp_{A^*}a^*)g(v_1,v_2)
=g(\mathcal{D}_{(a,a^*)}v_1,v_2)+g(v_1,\mathcal{D}_{(a,a^*)}v_2),$$
which means that both $\nabla$ and $\nabla^*$ preserve $g$.

In the particular case $A=TM$ with the Lie bracket, $A^*=T^*M$ with zero anchor
and zero bracket we simply speak of a generalized connection and $\nabla^*$ is a
tensor. Furthermore, if  $V=\mathbf{T}M$ and if we are interested in generalized
connections that preserve $g$ and $G$, $\nabla$ must be a $G$-metric big
connection and the tensor $\nabla^*$ must satisfy the conditions
$$
g(\nabla^*_\alpha\mathcal{X},\mathcal{Y})+
g(\mathcal{X},\nabla^*_\alpha\mathcal{Y})=0,\;
G(\nabla^*_\alpha\mathcal{X},\mathcal{Y})+
G(\mathcal{X},\nabla^*_\alpha\mathcal{Y})=0.
$$ It follows that $\nabla^*$
commutes with $\phi$ and preserves the subbundles $V_\pm$, therefore, $\nabla^*$
is equivalent with a pair $\Lambda^\pm$ of $\gamma$-skew-symmetric tensor fields of the type $(2,1)$ such that
$$ \nabla^*_\alpha(Y,\flat_{\psi\pm\gamma}Y)
=(\Lambda^\pm_\alpha Y,\flat_{\psi\pm\gamma}\Lambda^\pm_\alpha Y).
$$ If we denote
$$\Xi^\pm(X,Y,Z)=\gamma(\Lambda^\pm_{\flat_{\psi\pm\gamma}Z}X,Y),$$
the skew-symmetry condition becomes
$$\Xi^\pm(X,Y,Z)+\Xi^\pm(Y,X,Z)=0.$$ (Notice that the non-degeneracy of $\gamma$
implies the non-degeneracy of $\psi\pm\gamma$.)

For instance, we may take $\nabla=\nabla^{LC}$ to be the {\it big Levi-Civita
connection} defined by taking both $D^\pm$ equal to the Levi-Civita connection
$D$ and $\Xi^\pm=\pm d\psi$. The corresponding $\mathcal{D}^{LC}$ is the {\it
generalized Levi-Civita connection}. It codifies the same data like the
canonical big connection, but in a different way.

For generalized connections, the Courant torsion and the (totally skew
symmetric) Gualtieri torsion may be defined like for big connections, using the
operator $\mathcal{D}$ instead of $\nabla$, and we have
$$
T^\mathcal{D}((X,\alpha),(Y,\mu))= T^\nabla((X,\alpha),(Y,\mu))
+\nabla_\alpha^*(Y,\mu)-\nabla_\mu^*(X,\alpha).$$ Then, after some
calculations we get
$$
\begin{array}{l} \mathcal{T}^{ \mathcal{D}^{LC}}( \mathcal{X}_\pm,
\mathcal{Y}_\pm,\mathcal{Z}_\pm)= \mathcal{T}^{ \nabla^{LC}}( \mathcal{X}_\pm,
\mathcal{Y}_\pm,\mathcal{Z}_\pm)\pm 3d\psi(X,Y,Z),\vspace*{2mm}\\
\mathcal{T}^{ \mathcal{D}^{LC}}( \mathcal{X}_\pm,
\mathcal{Y}_pm,\mathcal{Z}_\mp)= \mathcal{T}^{ \nabla^{LC}}( \mathcal{X}_\pm,
\mathcal{Y}_\pm,\mathcal{Z}_mp)\vspace*{2mm}\\ \hspace*{3.3cm}\pm
d\psi(X,Y,\sharp_{\psi\pm\gamma}\flat_{\psi\mp\gamma}Z),\end{array}$$
$\forall\mathcal{X}_\pm,\mathcal{Y}_\pm,\mathcal{Z}_\pm\in V_\pm$. Thus, to end
the computation of the Gualtieri torsion of the generalized Levi-Civita
connection, we have to compute $ \mathcal{T}^{\nabla^{LC}}$ and technical
calculations that use the bracket formulas (\ref{auxgeneral}), (\ref{CptPpm})
yield
$$\mathcal{T}^{\nabla^{LC}}(
\mathcal{X},\mathcal{Y},\mathcal{Z})=-
d\psi(pr_{TM}\mathcal{X},pr_{TM}\mathcal{Y},pr_{TM}\mathcal{Z}),$$
$\forall\mathcal{X},\mathcal{Y},\mathcal{Z}\in\Gamma\mathbf{T}(M)$.

We may also define a curvature tensor. The correction that led to the Gualtieri
torsion may be seen as the use of a {\it modified Courant bracket}
$[\mathcal{X},\mathcal{Y}]^\mathcal{D}$ defined by
$$
g([\mathcal{X},\mathcal{Y}]^\mathcal{D},\mathcal{Z})=
g([\mathcal{X},\mathcal{Y}],\mathcal{Z})-
\frac{1}{2}g( \mathcal{D}_\mathcal{Z} \mathcal{X},\mathcal{Y})
+\frac{1}{2}g( \mathcal{D}_\mathcal{Z} \mathcal{Y},\mathcal{Z})
$$ (then,
$[\mathcal{X}_+,\mathcal{Y}_-]^\mathcal{D} =[\mathcal{X}_+,\mathcal{Y}_-]$ for
all $ \mathcal{X}_+\in\Gamma V_+,
\mathcal{Y}_-\in\Gamma V_-$). The formula
$$ \mathcal{R}^\mathcal{D}(
\mathcal{X},\mathcal{Y}) \mathcal{Z}=\mathcal{D}_\mathcal{X}
\mathcal{D}_\mathcal{Y} \mathcal{Z}-\mathcal{D}_\mathcal{Y}
\mathcal{D}_\mathcal{X}
\mathcal{Z}-\mathcal{D}_{[\mathcal{X},\mathcal{Y}]^\mathcal{D}} \mathcal{Z}
$$ defines a tensor that may be called the {\it generalized curvature tensor}.
\begin{rem}\label{obstwist}
{\rm If $M$ is endowed with a twisted Courant bracket
$$
[(X,\alpha),(Y,\mu)]^\Theta=([X,Y],L_X\mu-L_Y\alpha
+\frac{1}{2}d(\alpha(Y)-\mu(X)-i(Y)i(X)\Theta))$$ where $\Theta$ is
a closed $3$-form \cite{SW}, we may define a {\it twisted canonical big
connection} and a {\it twisted generalized Levi-Civita connection} in the same
way but replacing the form $d\psi$ with $d\psi+\Theta$. Then, we get a {\it
twisted Courant and Gualtieri torsion} $T^\mathcal{D}_\Theta(
\mathcal{X},\mathcal{Y})$, $\mathcal{T}^\mathcal{D}_\Theta(
\mathcal{X},\mathcal{Y},\mathcal{Z})$ replacing the Courant
bracket by the twisted Courant bracket.}\end{rem}
\section{Parallel Dirac structures}
Before referring to a single Dirac structure, we look at pairs of transversal
structures. Consider an endomorphism $\Psi\in End(\mathbf{T}M)$ such that
\begin{equation}\label{skewsym}  \Psi^2=\epsilon
Id,\;\;g(\mathcal{X},\Psi\mathcal{Y}) +
g(\Psi\mathcal{X},\mathcal{Y})=0,\hspace{2mm}
\epsilon=\pm1.\end{equation}
Then, the expression
\begin{equation}\label{NijPsi}
\mathcal{N}_\Psi(\mathcal{X},\mathcal{Y}) = [\Psi\mathcal{X},
\Psi\mathcal{Y}] -\Psi[\mathcal{X},\Psi\mathcal{Y}] -
\Psi[\Psi\mathcal{X},\mathcal{Y}]
+\Psi^2[\mathcal{X},\mathcal{Y}]=0
\end{equation}
has a tensorial character and it is called the {\it Courant-Nijenhuis torsion}
of $\Psi$. If $\epsilon=-1$, $\Psi$ is a generalized, almost complex structure $
\mathcal{J}$. If $\epsilon=1$, $\Psi$ is a {\it generalized, almost paracomplex
structure}. In both cases, if $
\mathcal{N}_\Psi=0$, the structure is {\it integrable} and the term
``almost" is omitted.

We refer to \cite{{Galt},{IV}} for the basics. In the complex case $\Psi$ may be
identified with the pair $E,\bar{E}$ of complex conjugate, transversal, almost
Dirac structures defined by its $\pm
\sqrt{-1}$-eigenbundles. In the paracomplex case $\Psi$ may be
identified with the pair $E,E'$ of real, transversal, almost Dirac structures
defined by its $\pm 1$-eigenbundles. In both cases, integrability is equivalent
with the property that the eigenbundles are integrable, i.e., closed under
Courant brackets. $\Psi$ has a representation by classical tensor fields:
\begin{equation}\label{matriceaPsi} \Psi\left(
\begin{array}{c}X\vspace{2mm}\\ \alpha \end{array}
\right) = \left(\begin{array}{cc} A&\sharp_\pi\vspace{2mm}\\
\flat_\sigma&-^t\hspace{-1pt}A\end{array}\right)
\left( \begin{array}{c}X\vspace{2mm}\\
\alpha \end{array}\right),\end{equation}
where $A\in End (TM),\pi\in\chi^2(M),\sigma\in\Omega^2(M)$, $t$ denotes
transposition and
$$ A^2=\epsilon Id -
\sharp_\pi\circ\flat_\sigma,\;\pi(\alpha\circ A,\beta)=\pi(\alpha,
\beta\circ A),\;\sigma(AX,Y)=\sigma(X,AY).$$ The expression of the integrability condition in terms of $(A,\pi,\sigma)$ is known and it includes the fact that $\pi$ is a Poisson bivector field.

We are interested in structures $\Psi$ on a generalized Riemannian manifold
$(M,G)$. Then, $\Psi$ is {\it compatible} with $G$ if
\begin{equation}\label{compatGPhi0} G(\Psi\mathcal{X},\Psi\mathcal{Y})
= G(\mathcal{X},\mathcal{Y}), \end{equation} equivalently,
\begin{equation}\label{compatGPhi}
\phi\circ\Psi=-\epsilon\Psi\circ\phi, \end{equation}
where $\phi$ is defined by (\ref{eqGg}). A compatible pair $(G,\Psi)$ with
$\epsilon=-1$, respectively $\epsilon=1$, is an {\it almost generalized
Hermitian}, respectively {\it para-Hermitian}, structure and ``almost" is
omitted in the integrable case. In the Hermitian case, condition
(\ref{compatGPhi0}) is equivalent with the fact that the complex subbundles
$E,\bar{E}$ are $G$-isotropic. In the para-Hermitian case, condition
(\ref{compatGPhi0}) is equivalent with the fact that the eigenbundles
$E,E'$ are $G$-orthogonal.

In the Hermitian case, (\ref{compatGPhi}) shows that $\mathcal{J}=\Psi$ preserves the
eigenbundles $V_\pm$, hence, it corresponds bijectively with a pair of
$\gamma$-compatible, almost complex structures $J_\pm$ of $M$ obtained by the
transfer of $\mathcal{J}|_{V_\pm}$ to $TM$ via the isomorphisms $\tau_\pm$ of (\ref{isotau}). In other
words, $\mathcal{J}$ is expressed by
$$
\mathcal{J}(X,\flat_{\psi\pm\gamma}X)=(J_\pm X,\flat_{\psi\pm\gamma}J_\pm
X).$$

In the para-Hermitian case, (\ref{compatGPhi}) shows that $\Psi$ interchanges
the eigenbundles $V_\pm$ and $\Psi$ bijectively corresponds to a bundle
isomorphism $F$ of $TM$ such that
\begin{equation}\label{defF}
\Psi(X,\flat_{\psi+\gamma}X)=(FX,\flat_{\psi-\gamma}FX)\end{equation}
and $F$ satisfies the condition
\begin{equation}\label{gammaF}
\gamma(FX,FY)=\gamma(X,Y).\end{equation}
By replacing $X$ with $F^{-1}X$, we get
$$
\Psi(X,\flat_{\psi-\gamma}X)=(F^{-1}X,\flat_{\psi+\gamma}F^{-1}X),
\hspace{3mm}\gamma(F^{-1}X,F^{-1}Y)=\gamma(X,Y).
$$

If we express $\Psi$ of (\ref{defF}) by (\ref{matriceaPsi}), we get
$$ \begin{array}{ll}
F=A+\sharp_\pi\circ\flat_{\psi+\gamma},&
\flat_{\psi-\gamma}\circ F=\flat_\sigma-
\hspace{1pt}^t\hspace{-1pt}A\circ\flat_{\psi+\gamma},\vspace*{2mm}\\
F^{-1}=A+\sharp_\pi\circ\flat_{\psi-\gamma},&
\flat_{\psi+\gamma}\circ F^{-1}=\flat_\sigma-
\hspace{1pt}^t\hspace{-1pt}A\circ\flat_{\psi-\gamma}.\end{array}
$$
Then, by addition and subtraction:
\begin{equation}\label{FA} \begin{array}{c}
\sharp_\pi=\frac{1}{2}(F-F^{-1})\circ\sharp_\gamma,\;
A=\frac{1}{2}(F+F^{-1})-\sharp_\pi\flat_\psi,\vspace*{2mm}\\
\flat_\sigma=\flat_\psi\circ(F+F^{-1})-\hspace{1pt}^t\hspace{-1pt}A\circ\flat_\psi.
\end{array}\end{equation}

Furthermore, since the projections $(1/2)(Id\pm\Psi)$ restrict to isomorphisms
$V_+\rightarrow E,V_-\rightarrow E'$, we have
\begin{equation}\label{E} \begin{array}{l}
E=\{(X,\flat_{\psi+\gamma}X)+(FX,\flat_{\psi-\gamma}FX)\},\vspace{2mm}\\
E'=\{(X,\flat_{\psi+\gamma}X)-(FX,\flat_{\psi-\gamma}FX)\},\end{array}
\end{equation}  where the
representation of the elements of $E,E'$ is unique.

Now, we shall address the question of integrability and we start with the
following result.
\begin{prop}\label{existccompat} On any (para-)Hermitian manifold
$(M,\Psi)$ there are compatible, $G$-metric, big connections.
\end{prop}
\begin{proof} For a big connection $\nabla$, compatibility means
$\nabla\circ\Psi=\Psi\circ\nabla$. Using the expressions of $\nabla$ and
$\Psi$ on $V_\pm$, we see that, in the Hermitian case, the compatibility
condition is equivalent with
\begin{equation}\label{compatcinH} D^\pm\circ J_\pm=J_\pm\circ
D^\pm.\end{equation} Since there exist many $\gamma$-metric connections that
satisfy (\ref{compatcinH}) (connections on the unitary principal bundles of
frames associated with $(\gamma,J_\pm)$), the required existence result holds.

In the para-Hermitian case, the compatibility condition reduces to
\begin{equation}\label{compatcinpH} D^-\circ F=F\circ D^+,
\end{equation} which implies the second required condition
$D^+\circ F^{-1}=F^{-1}\circ D^-$ because $F$ is an isomorphism. Locally, pairs
of connections satisfying  (\ref{compatcinpH}) exist (take a local basis $(e_i)$
of $TM$ and put $D^+e_i=0,D^-(Fe_i)=0$). Then, corresponding global pairs can be
constructed by the usual gluing procedure with a partition of unity.
\end{proof}
\begin{rem}\label{compatgen} {\rm Compatibility of $\Psi$ with a generalized
connection is a more complicated condition since it adds the requirement
$\nabla^*_\alpha\circ\Psi=\Psi\circ\nabla^*_\alpha$. The compatibility of the
generalized Levi-Civita connection with $\Psi$ requires the conditions
$$ d\psi(J_\pm X,Y,Z)=
-d\psi(X,J_\pm Y,Z),$$
$$
d\psi(FX,FY,Z)=-d\psi(X,Y,\sharp_{\psi+\gamma}\flat_{\psi-\gamma}Z),
$$ respectively for $\epsilon=\pm1$. The first condition holds iff $d\psi=0$
(check on arguments in the eigenspaces of $J_\pm$ and use the skew-symmetry of
$d\psi$). The second condition follows by using
$$\alpha=\flat_{\psi+\gamma}\sharp_{\psi+\gamma}\alpha,\; Z=\sharp_{\psi-\gamma}\alpha$$
and replacing $Y$ by $FY$.}\end{rem}

The next result that we need is
\begin{prop}\label{integrcuconex} If $\nabla$ is a big connection that
commutes with the integrable, generalized almost (para-)complex structure
$\Psi$, the Gualtieri torsion of $\nabla$ satisfies the condition
\begin{equation}\label{tGintegr}
\epsilon\mathcal{T}^\nabla(\mathcal{X},\mathcal{Y},\mathcal{Z})
+\mathcal{T}^\nabla(\mathcal{X},\Psi\mathcal{Y},\Psi\mathcal{Z})
+\mathcal{T}^\nabla(\Psi\mathcal{X},\mathcal{Y},\Psi\mathcal{Z})
\end{equation}
$$+\mathcal{T}^\nabla(\Psi\mathcal{X},\Psi\mathcal{Y},\mathcal{Z})=0.$$
Conversely, if there exists a big connection that commutes with $\Psi$ and
satisfies (\ref{tGintegr}), $\Psi$ is integrable.\end{prop}
\begin{proof}
A straightforward calculation \cite{G2007} shows that the Courant-Nijenhuis
torsion of $\Psi$ and the Gualtieri torsion of a $\Psi$-compatible big
connection $\nabla$ are related by the following formula
\begin{equation}\label{NijcutG}
g(\mathcal{N}_{\Psi}(\mathcal{X},\mathcal{Y}),\mathcal{Z})+
\epsilon\mathcal{T}^\nabla(\mathcal{X},\mathcal{Y},\mathcal{Z})
+\mathcal{T}^\nabla(\mathcal{X},\Psi\mathcal{Y},\Psi\mathcal{Z})
\end{equation}
$$+\mathcal{T}^\nabla(\Psi\mathcal{X},\mathcal{Y},\Psi\mathcal{Z})
+\mathcal{T}^\nabla(\Psi\mathcal{X},\Psi\mathcal{Y},\mathcal{Z})=0.$$
\end{proof}
\begin{rem}\label{Nijcugenc}
{\rm The conclusions of Proposition \ref{integrcuconex} hold if either we replace
$\nabla$ by a generalized connection $\mathcal{D}$ or we replace the
Nijenhuis and Courant torsions by the twisted Nijenhuis and Courant torsions
produced by the twisted Courant bracket.} \end{rem}

Proposition \ref{integrcuconex} implies the fact that a generalized K\"ahler
structure $(G,\mathcal{J})$ (see \cite{Galt} for the definition) is
characterized by the following couple of properties
\cite{G2007}: (a) the canonical big connection $\nabla$ commutes with
the generalized almost complex structure $\mathcal{J}$, (b) the Gualtieri
torsion of the canonical connection is a sum of components of $
\mathcal{J}$-type $(2,1)$ and $(1,2)$. Indeed, property (a) is equivalent with
$D^\pm J_\pm=0$ and, if the notion of  $\mathcal{J}$-type is defined like for usual complex structures,
property (b) is equivalent with (\ref{tGintegr}) and, further, with the fact that $d\psi$ is a sum of
components of $J_\pm$-type $(2,1)$ and $(1,2)$. These two properties
characterize the generalized K\"ahler structures \cite{Galt}. By Remark
\ref{Nijcugenc} and with the results of \cite{Galt}, a similar
characterization holds for twisted generalized K\"ahler structures, if the
canonical connection is defined in accordance to Remark \ref{obstwist}.

Now, let us consider the following situation
\begin{defin}\label{canparalel} {\rm A generalized para-Hermitian structure
$(G,\Psi)$ is said to be {\it parallel} if the (integrable) structure $\Psi$ commutes with the canonical big connection $\nabla$ of $G$.}\end{defin}
\begin{prop}\label{propcanpar} The generalized para-Hermitian structure
$(G,\Psi)$ is parallel iff $d\psi=0$ and the $\gamma$-isometry $F$ that defines
$\Psi$ is parallel with respect to the Levi-Civita connection $D$ of
$\gamma$.\end{prop}
\begin{proof} Since the canonical big connection has no mixed torsion,
by looking at three arguments in the same subbundle $V_\pm$, we see that
(\ref{tGintegr}) is equivalent with $d\psi=0$. Then, $D^\pm=D$ and the
commutation condition (\ref{compatcinpH}) becomes $DF=0$.\end{proof}
\begin{rem}\label{obsstrpi} {\rm From (\ref{FA}), it follows
that a parallel structure has an associated, Levi-Civita parallel, Poisson
bivector field $\pi$. Hence, by a result of Lichnerowicz (e.g., see
\cite{V-carte}, Proposition 3.12) $\gamma$ is a decomposable metric with a
K\"ahlerian component.}\end{rem}
\begin{example}\label{infinitate} {\rm Let $(M,\gamma,J)$ be a K\"aher manifold. Then, $F=J$ is parallel and, for any generalized Riemannian
metric $G$ defined by $\gamma$ and by a closed $2$-form $\psi$, we get a
parallel structure $\Psi$, namely,
\begin{equation}\label{F=J} \Psi(X,\flat_{\psi\pm\gamma}X)=\pm(JX,
\flat_{\psi\mp\gamma}JX).\end{equation}
For the structure (\ref{F=J}), formulas (\ref{E}) yield $$E=graph\,
\flat_{\psi-\omega}, E'=graph\,\flat_{\psi+\omega},$$
where $\omega(X,Y)=\gamma(JX,Y)$ is the K\"ahler form of
$(\gamma,J)$.}\end{example}

Now, let us consider a single almost Dirac structure $E$ on a generalized
Riemannian manifold $(M,G)$. Obviously, $E$ may be identified with the unique,
$G$-compatible, generalized, paracomplex structure $\Psi_E$ of $+1$-eigenbundle
$E$ and $-1$-eigenbundle $E'=E^{\perp_G}=\phi(E)$ (the last equality follows
from (\ref{eqGg}) and (\ref{Riemannbig2})). We will denote by $F_E$ the isometry
of the bundle $(TM,\gamma)$ that corresponds to $\Psi_E$.

A first expression of the integrability condition of $E$ is
$$g([\mathcal{X},\mathcal{Y}],\mathcal{Z})=0,\hspace{2mm}
\mathcal{X},\mathcal{Y},\mathcal{Z}\in\Gamma E,$$ where $E$ is given by
(\ref{E}). If we denote $F_1=Id+F_E,F_2=Id-F_E$, which yields
$$\gamma(F_1X,F_2Y)=-\gamma(F_1Y,F_2X)$$ because of (\ref{gammaF}), we may write
$$ \mathcal{X}=(F_1X,\flat_\psi F_1X+\flat_\gamma F_2X),\,
\mathcal{Y}=(F_1Y,\flat_\psi F_1Y+\flat_\gamma F_2Y),\,
\mathcal{Z}=(F_1Z,\flat_\psi F_1Z+\flat_\gamma F_2Z),$$
where $X,Y,Z$ are vector fields on $M$. Then, using formula (\ref{Cptpsi}) and
making the required technical computations, the integrability condition of $E$
becomes
$$\sum_{Cycl(X,Y,Z)}\{(F_1X)\gamma(F_1Y,F_2Z)-\gamma([F_1X,F_1Y],F_2Z)\}
=d\psi(F_1X,F_1Y,F_1Z).$$

Below, we show another way to express the integrability of $E$. For any
generalized paracomplex structure $\Psi$, we may define the {\it
Courant-Ehresmann curvature} of $E$ with respect to $E'$ by
$$ \mathcal{E}_{(E;E')}(\mathcal{X},
\mathcal{Y})=(Id-\Psi)[(Id+\Psi)\mathcal{X},
(Id+\Psi)\mathcal{Y}].$$ Then, we get
$$\mathcal{E}_{(E';E)}+\mathcal{E}_{(E;E')}=\mathcal{N}_\Psi,
\;\mathcal{E}_{(E';E)}-\mathcal{E}_{(E;E')}=\Psi\mathcal{N}_\Psi,$$
therefore:
\begin{equation}\label{Ehr2} \mathcal{E}_{(E;E')}=\frac{1}{2}(Id-\Psi)\mathcal{N}_\Psi.
\end{equation}
Obviously, $E$ is integrable iff $\mathcal{E}_{(E;E')}=0$.

Our next remark is that a $G$-metric big connection preserves the almost Dirac
structure $E$ (i.e., $\nabla_X\mathcal{Y}\in\Gamma E$,
$\forall\mathcal{Y}\in\Gamma E)$ iff $\nabla$ commutes with $\Psi_L$ and by Proposition \ref{existccompat},
such connections exist for every $E$. Then, we get
\begin{prop}\label{integrLcutG} If there exists a big connection
$\nabla$ that preserves $E$ and is such that the Gualtieri torsion satisfies the
condition
\begin{equation}\label{tGLintegr}
\mathcal{T}^\nabla(\mathcal{X},\mathcal{Y},\mathcal{Z})
+\mathcal{T}^\nabla(\mathcal{X},\Psi_E\mathcal{Y},\mathcal{Z})
+\mathcal{T}^\nabla(\Psi_E\mathcal{X},\mathcal{Y},\mathcal{Z})
\end{equation}
$$+\mathcal{T}^\nabla(\Psi_E\mathcal{X},\Psi_L\mathcal{Y},\mathcal{Z})=0,
\hspace{3mm}\forall\mathcal{X},\mathcal{Y}\in\Gamma \mathbf{T}M,\mathcal{Z}
\in\Gamma E,$$ then, the almost Dirac structure $E$
is integrable Conversely, if $E$ is integrable, (\ref{tGLintegr}) holds for any big connection $\nabla$ that preserves $E$.\end{prop}
\begin{proof} Use (\ref{Ehr2}) and insert $\mathcal{N}_{\Psi}$ as given by (\ref{NijcutG}) in
the integrability condition $\mathcal{E}_{(E;E')}=0$. Then, use $E=im(Id+\Psi)$ and $\Psi|_E=Id.$
\end{proof}
\begin{defin}\label{Dircanpar} {\rm The Dirac structure $E$ is
{\it parallel} on $(M,G)$ if the canonical big connection $\nabla$ of $G$
preserves $E$.}\end{defin}
\begin{prop}\label{propDirpar} The almost Dirac structure $E$
is a parallel Dirac structure iff the following two conditions hold:
\begin{equation}\label{LCDirpar} \gamma(F_EZ,D_XF_E(Y))=
\frac{1}{2}[d\psi(X,Y,Z)+d\psi(X,F_EY,F_EZ)],\end{equation}
\begin{equation}\label{psiDirpar} d\psi(X,Y,Z)+d\psi(F_EX,F_EY,F_EZ)=0,
\end{equation}
where $D$ is the Levi-Civita connection of $\gamma$.\end{prop}
\begin{proof} With a few simple technicalities,
(\ref{LCDirpar}) follows from the expression of the commutation condition
(\ref{compatcinpH}) for the connections (\ref{DpmcuLC}). Then, if we replace
$\mathcal{Z}$ by $(Id+\Psi_E)\mathcal{Z}$ in (\ref{tGLintegr}), and consider the
result for all possible combinations of arguments in $V_\pm$ while remembering
that the canonical big connection has no mixed torsion and satisfies
(\ref{Gtorsnabla2}), we see that the only condition required for the
integrability of $E$ is (\ref{psiDirpar}).
\end{proof}
\begin{example}\label{Psparalel} {\rm Take $E=graph\,\sharp_P$,
$P\in\chi^2(M)$. If we express $(\sharp_P\alpha,\alpha)$ by the first formula
(\ref{E}), we get
$$X+F_EX=\sharp_P\alpha,\;\flat_\psi(X+F_EX)+\flat_\gamma(X-F_EX)=\alpha,$$ which leads to
$$X=\frac{1}{2}\sharp_\gamma(\alpha-\flat_{\psi-\gamma}\sharp_P\alpha),
\;F_EX=-\frac{1}{2}\sharp_\gamma(\alpha-\flat_{\psi+\gamma}\sharp_P\alpha).$$
Thus, $Id-\flat_{\psi-\gamma}\sharp_P$ must be an isomorphism, which we may also
see as follows. $\forall U\in TM$ we have
$$<(Id-\flat_{\psi-\gamma}\sharp_P)\flat_{\psi+\gamma}U,U>=
(\psi+\gamma)(U,U)+<\sharp_P\flat_{\psi+\gamma}U,\flat_{\psi+\gamma}U>
=\gamma(U,U),$$ which vanishes only for $U=0$. Then, since $\psi+\gamma$ is non degenerate, $<(Id-\flat_{\psi-\gamma}\sharp_P)\flat_{\psi+\gamma}U=0$ iff $U=0$, and we are done. Similarly, $Id-\flat_{\psi+\gamma}\sharp_P$ is an
isomorphism.

In the previous expressions of $X,F_EX$ it is preferable to replace
$\alpha$ by $\flat_\gamma\sharp_\gamma\alpha$. Accordingly, the isometry $F_E$
gets the form
$$ F_EX=(Q^+-Id)(Q^-+Id)^{-1}X\hspace{5mm} (Q^\pm=\pm\sharp_\gamma\flat_{\psi\pm\gamma}\sharp_P\flat_\gamma).
$$
Then, we may check that $F_E$ satisfies condition (\ref{gammaF}) by writing
down the latter for $(Q^-+Id)X,(Q^-+Id)Y$ instead of $X,Y$ and taking into
account the skew symmetry of $\psi$ and $P$.

Now, if we replace the arguments $Y,Z$ by $(Q^-+Id)Y, (Q^-+Id)Z$ in
(\ref{LCDirpar}) and $X,Y,Z$ by $(Q^-+Id)X,(Q^-+Id)Y, (Q^-+Id)Z$ in
(\ref{psiDirpar},) we get the following characteristic conditions for
$graph\,\sharp_P$ to be parallel on $(M,G)$
\begin{equation}\label{LPDirpar} \begin{array}{l}
\gamma((Q^+-Id)Z,(D_XF_E)(Q^+-Id)Y)\vspace*{2mm}\\ =
\frac{1}{2}[d\psi(X,(Q^-+Id)Y,(Q^-+Id)Z)\vspace*{2mm}\\
+d\psi(X,(Q^+-Id)Y,(Q^+-Id)Z)],\vspace*{2mm}\\
d\psi((Q^-+Id)X,(Q^-+Id)Y,(Q^-+Id)Z)\vspace*{2mm}\\
+d\psi((Q^+-Id)X,(Q^+-Id)Y,(Q^+-Id)Z)=0.\end{array}
\end{equation}
If $\psi=0$, then, $Q^+=Q_-=Q$ and (\ref{LPDirpar}) reduce to $DF_E=0$,
equivalently,
$$DQ=(Q-Id)(Q+Id)^{-1}DQ.$$ Putting $(Q+Id)^{-1}DQ=S,$ the previous
condition becomes $$(Q+Id)S=(Q-Id)S,$$ i.e., $S=0$. But, $S=0$ iff
$D\sharp_P=0$. Thus, in the classical case, the graph of $P$ is parallel iff $P$
is a $\gamma$-parallel Poisson structure.}\end{example}
\section{Symplectic Dirac structures}
In this section we shall discuss a special kind of Dirac structures that appear
in connection with endomorphisms $\tau\in End (\mathbf{T}M)$ such that
(\ref{skewsym}) with $\Psi=\tau$ and $\epsilon=0$ holds, i.e.,
\begin{equation}\label{skewsymtg}  \tau^2=0\;\;g(\mathcal{X},\tau\mathcal{Y})
+g(\tau\mathcal{X},\mathcal{Y})=0.\end{equation}

In \cite{IV} such endomorphisms were called generalized subtangent structures.
In the present paper, a generalized subtangent structure of constant rank will
be called a {\it generalized 2-nilpotent structure}. If $rank\,\tau=dim\,M$, we
stick with the terminology of \cite{IV} and call $\tau$ a {\it generalized
almost tangent structure}.

By (\ref{skewsymtg}), the image $E=im\,\tau$ of a generalized, 2-nilpotent structure is a $g$-isotropic subbundle, i.e., a
big-isotropic structure in the sense of \cite{V2}. We denote by
$E^{\perp_g}$ the $g$-orthogonal subbundle of $E$ and notice that (\ref{skewsymtg}) implies
$E^{\perp_g}\subseteq ker\,\tau$. Moreover, since these subbundles have the same rank, we have
$E^{\perp_g}= ker\,\tau$. This remark leads to the existence of a well defined,
non-degenerate $2$-form $\omega\in\Gamma\wedge^2E^*$ ($E^*$ is the dual bundle
of $E$) given by
\begin{equation}\label{omega}
\omega(e_1,e_2)=g(e_1,\mathcal{X}_2),\hspace{2mm}e_1,e_2\in\Gamma
E,\,\tau\mathcal{X}_2=e_2\end{equation} (independent of the choice of $
\mathcal{X}_2$). The converse is also true, i.e., if $E$ is a big-isotropic
structure and $\omega\in\Gamma\wedge^2E^*$ is non degenerate, (\ref{omega})
uniquely defines an element $e_2=\tau\mathcal{X}_2\in E$ and we see that there
exists a unique generalized 2-nilpotent structure $\tau$ with $E=im\,\tau$ and with the
given form $\omega$. The non-degeneracy of $\omega$ implies the fact that a generalized
2-nilpotent structure $\tau$ has an even rank. It also follows that a
big-isotropic structure $E$ is the image of a generalized 2-nilpotent structure
$\tau$ iff the structure group of $E$ is reducible to a symplectic group.

Put $ \tilde{E}=\mathbf{T}M/E^{\perp_g}$. Since $E^{\perp_g}= ker\,\tau$,
$\tau$ induces an isomorphism $\tau':\tilde{E}\rightarrow E$ given by
$\tau'\mathcal{X}_{mod E^{\perp_g}}=\tau\mathcal{X}$. On $\tilde{E}$ we have the
skew-symmetric, non-degenerate $2$-form $\Lambda$ defined by
$$ \Lambda(\mathcal{X}_{mod
E^{\perp_g}}, \mathcal{Y}_{mod
E^{\perp_g}})=g(\tau\mathcal{X},\mathcal{Y}).$$ The quotient bundle
$\tilde{E}$ is canonically isomorphic to the dual bundle $E^*$ by means of the
pairing
$$<\mathcal{X}_{mod\,E^{\perp_g}},\mathcal{Y}>=g(\mathcal{X},\mathcal{Y})
\hspace{3mm}( \mathcal{X}\in\Gamma \mathbf{T}M,\mathcal{Y}\in\Gamma
E).$$ Thus, $\Lambda$ may be seen as a bivector field of $E$. Recall the musical
isomorphisms $\flat_\omega:E\rightarrow E^*$, $\flat_\omega e=i(e)\omega$, and
$\sharp_\Lambda:E^*\rightarrow E$, $\sharp_\Lambda\mathcal{X}_{mod
E^{\perp_g}}=i(\mathcal{X}_{mod E^{\perp_g}})\Lambda$. From the given
definitions we see that $\flat_\omega e\in\Gamma E^*$ identifies with $
\mathcal{Y}_{mod\,E^{\perp_g}}$ such that $\tau\mathcal{Y}=-e$. On the other hand,
we have $\sharp_\Lambda\mathcal{X}_{mod E^{\perp_g}}=\tau'\mathcal{X}_{mod
E^{\perp_g}}=\tau\mathcal{X}$, therefore, $\sharp_\Lambda\circ\flat_\omega=-Id$
and $\Lambda(\flat_\omega e_1,\flat_\omega e_2)=\omega(e_1,e_2)$.
\begin{defin}\label{slabint} {\rm The generalized 2-nilpotent
structure $\tau$ is {\it weakly integrable} if the big-isotropic structure
$E=im\,\tau$ is integrable, i.e., closed under Courant brackets. The generalized
2-nilpotent structure $\tau$ is {\it integrable} if its Courant-Nijenhuis
torsion is $\mathcal{N}_\tau=0$. If $\tau$ is integrable and
$rank\,\tau=dim\,M$, then $\tau$ is a {\it generalized tangent
structure}.}\end{defin}

Thus, $\tau$ is weakly integrable iff,
$\forall\,\mathcal{X},\mathcal{Y}\in\Gamma \mathbf{T}M$, one has
$[\tau\mathcal{X},\tau\mathcal{Y}]\in\Gamma E$, equivalently, the formula
\begin{equation}\label{2slabint}
[ \mathcal{X}_{{\rm mod}\,E^{\perp_g}},\mathcal{Y}_{{\rm
mod}\,E^{\perp_g}}]_{\tilde{E}}=\tau^{'-1}[\tau\mathcal{X},\tau\mathcal{Y}]
\end{equation}
yields a well defined new bracket on $\Gamma\tilde{E}$. Then,
$(\tilde{E},pr_{TM}\circ\tau',[\,,\,]_{\tilde{E}})$ is a Lie algebroid.

On the other hand, (\ref{NijPsi}) shows that $\tau$ is weakly integrable iff
$\mathcal{N}_\tau(
\mathcal{X},\mathcal{Y})\in\Gamma E$, therefore, integrability
implies weak integrability. In the almost tangent case, we have $E^{\perp_g}=E$
and the weak integrability condition becomes
$$ \tau\circ\mathcal{N}_\tau=0.$$
\begin{prop}\label{slab-tare} The generalized, 2-nilpotent
structure $\tau$ is integrable iff it is weakly integrable and $\omega$ is a
symplectic form of the Lie algebroid $E$.\end{prop}
\begin{proof}
If $d_E$ denotes the exterior differential of the Lie algebroid $E$, we have
\begin{equation}\label {auxlema} \begin{array}{l}
d_E\omega(\tau\mathcal{X},\tau\mathcal{Y},\tau\mathcal{Z})=
pr_{TM}\tau\mathcal{X}(g(\tau\mathcal{Y},\mathcal{Z}))
-pr_{TM}\tau\mathcal{Y}(g(\tau\mathcal{X},\mathcal{Z}))\vspace*{2mm}\\
+pr_{TM}\tau\mathcal{Z}(g(\tau\mathcal{X},\mathcal{Y}))
-g([\tau\mathcal{X},\tau\mathcal{Y}],\mathcal{Z})+
g([\tau\mathcal{X},\tau\mathcal{Z}],\mathcal{Y})\vspace*{2mm}\\
-g([\tau\mathcal{Y},\tau\mathcal{Z}],\mathcal{X}).\end{array}\end{equation} If
the general property \cite{LWX}
$$pr_{TM} \mathcal{X}(g( \mathcal{Y},\mathcal{Z}))=
g([ \mathcal{X},\mathcal{Y}]+\partial g(\mathcal{X},\mathcal{Y}) ,\mathcal{Z}) +
g(\mathcal{Y},[ \mathcal{X},\mathcal{Z}]+\partial g(\mathcal{X},\mathcal{Z})),$$
where $\partial$ is defined by
\begin{equation}\label{defpartial} pr_{TM}\mathcal{X}(f)=2g(\mathcal{X},\partial f)
\hspace{3mm}(f\in C^\infty(M)),\end{equation}
is applied to the second and third term of the right hand side of
(\ref{auxlema}), reductions lead to the formula
$$
d_E\omega(\tau\mathcal{X},\tau\mathcal{Y},\tau\mathcal{Z})=-g(
\mathcal{X},\mathcal{N}_\tau(
\mathcal{Y},\mathcal{Z})),$$
which proves the conclusion of the proposition.
\end{proof}

The proposition characterizes the big-isotropic and Dirac
structures that are images of an integrable, generalized, 2-nilpotent structure and we will call them
{\it symplectic big-isotropic and Dirac structures}.
\begin{rem}\label{Lamdapi} {\rm The symplectic structure of the Lie algebroid
$E$ defines a Poisson structure on $M$, which is given by
$$\{f,h\}=\Lambda(d_Ef,d_Eh),\;\;f,h\in C^\infty(M).$$
Using (\ref{defpartial}), it follows that $d_Ef,d_Eh$ are represented by
$2[\partial f]_{{\rm mod}\,E^{\perp_g}}, 2[\partial h]_{{\rm mod}\,E^{\perp_g}}$
in $ \tilde{E}$ and that $\partial f=(1/2)(0,df)$. Thus, with the definition of
$\Lambda$, we get
$$\{f,h\}=g(\tau(0,df),(0,dh))=\pi(df,dh),$$ where $\pi$ is the bivector field
of the matrix representation (\ref{matriceaPsi}).}\end{rem}
\begin{example}\label{expresympl} {\rm For any closed $2$-form $\theta$,
$$graph\,\flat_\theta=\{(X,\flat_\theta X)\,/\,X\in TM\}$$
is a Dirac structure. Since one has
$$[(X,\flat_\theta X),(Y,\flat_\theta
Y)]=([X,Y],\flat_\theta[X,Y])$$ (see (\ref{Cptpsi})), $(X,\flat_\theta X)\mapsto X$ is an isomorphism
between the Lie algebroids $E_\theta$ and $TM$, which identifies $2$-$E$-forms
with differential $2$-forms on $M$ and $d_E$ with $d$. Proposition
\ref{slab-tare} gives a bijection between the
generalized tangent structures $\tau$ on $M$ with image $graph\,\theta$ and
symplectic forms $\mu$ on $M$. With (\ref{omega}), we get the expression of this
correspondence:
$$\tau(X,\alpha)=(U,\flat_\theta U),\hspace{5mm} U=
\sharp_\mu(\alpha-\flat_\theta
X)\hspace{3mm}(\sharp_\mu\flat_\mu=-Id).$$ If the manifold $M$ has no symplectic
forms, the graph of a presymplectic form is not the image of a generalized
tangent structure.}\end{example}
\begin{example}\label{exPoisson} {\rm If $P$ is a Poisson bivector field,
$$graph\,P=\{(\sharp_P\alpha,\alpha)\,/\,\alpha\in T^*M\}$$ is a Dirac structure on $M$.
The Courant bracket within $graph\,P$ is
$$[(\sharp_P\alpha,\alpha),(\sharp_P\beta,\beta)]=
(\sharp_P\{\alpha,\beta\}_P,\{\alpha,\beta\}_P),$$ where the bracket of
$1$-forms is that of the Lie algebroid structure of $T^*M$ defined by $P$ (e.g.,
\cite{V-carte}). Therefore, the mapping $(\sharp_\Pi\alpha,\alpha)\mapsto\alpha$
is an isomorphism between the Lie algebroids $graph\,P$ and $T^*M$, and the
generalized tangent structures $\tau$ on $M$ with image $graph\,P$ are in a one-to-one correspondence with the non degenerate $2$-cocycles of the Lie algebroid $T^*M$,
i.e., the bivector fields $W$ on $M$ that satisfy the condition $[P,W]=0$
(Schouten-Nijenhuis bracket). Explicitly, the correspondence is given by
$$\tau(X,\alpha)=(\sharp_P\lambda,\lambda),\hspace{5mm}
\lambda=\flat_W(X-\sharp_P\alpha)\hspace{3mm}
(\flat_W\sharp_W=-Id).$$}\end{example}

In order to give another expression of the relation between integrability and
weak integrability we define the bracket
$$ [\mathcal{X},\mathcal{Y}]_\tau=
[\tau \mathcal{X},\mathcal{Y}]+[\mathcal{X},\tau
\mathcal{Y}],$$ which puts
the integrability condition $\mathcal{N}_\tau=0$ under the form
\begin{equation}\label{tauhomgen} \tau[\mathcal{X},\mathcal{Y}]_\tau=
[\tau \mathcal{X},\tau
\mathcal{Y}].\end{equation} Straightforward computations that use
(\ref{skewsymtg}) and the Courant algebroid axioms \cite{LWX} for $\mathbf{T}M$
give the following properties of the new bracket
\begin{equation}\label{taufgen}
[\mathcal{X},f\mathcal{Y}]_\tau= f[\mathcal{X},\mathcal{Y}]_\tau+pr_{TM}\tau
\mathcal{X}(f)\mathcal{Y}+pr_{TM}\mathcal{X}(f)\tau
\mathcal{Y},\end{equation}
\begin{equation}\label{tauJacobigen}  \begin{array}{lcl}
\sum_{Cycl(\mathcal{X},\mathcal{Y},\mathcal{Z})}
[[\mathcal{X},\mathcal{Y}]_\tau,\mathcal{Z}]_\tau&=&
\sum_{Cycl(\mathcal{X},\mathcal{Y},\mathcal{Z})}[\mathcal{Z},
\mathcal{N}_\tau(\mathcal{X},\mathcal{Y})]\vspace{2mm}\\
&+&\frac{1}{3}\partial\sum_{Cycl(\mathcal{X},\mathcal{Y},\mathcal{Z})}
g(\mathcal{Z},
\mathcal{N}_\tau(\mathcal{X},\mathcal{Y})).\end{array}
\end{equation}

Let us assume that $\tau$ is weakly integrable. Since the closure of $E$ under
Courant brackets is equivalent to $[\Gamma E,\Gamma E^{\perp_g}]\subseteq\Gamma
E^{\perp_g}$
\cite{V2}, we get $[
\mathcal{X},\mathcal{Y}]_\tau\in E^{\perp_g}$, $\forall\mathcal{X}\in\Gamma
\mathbf{T}M,\forall\mathcal{Y}\in\Gamma E^{\perp_g}$, and
$$
[\mathcal{X}_{mod E^{\perp_g}},\mathcal{Y}_{mod E^{\perp_g}}]_\tau=
[\mathcal{X},\mathcal{Y}]_{\tau,mod E^{\perp_g}}$$ is a well defined
bracket on the quotient bundle $\tilde{E}$, which we call the {\it induced
$\tau$-bracket}.
\begin{prop}\label{integr1} If the generalized 2-nilpotent structure
$\tau$ is integrable, then $\tilde{E}$ with the induced $\tau$-bracket and the
anchor $pr_{TM}\circ\tau'$ is a Lie algebroid and $\tau'$ is an isomorphism of
Lie algebroids. Furthermore, the weakly integrable, generalized, 2-nilpotent
structure $\tau$ is integrable iff the $\tau$-induced bracket of $\tilde{E}$ is equal
to the bracket $[\,,\,]_{\tilde{E}}$.\end{prop}
\begin{proof} For the first part of the proposition
check the axioms of a Lie algebroid using (\ref{taufgen}), (\ref{tauJacobigen})
and $ \mathcal{N}_\tau=0$ (the Lie algebroid $E$ has the usual Courant bracket
of $ \mathbf{T}M$). In the second part of the proposition, $[\,,\,]_{\tilde{E}}$
is the bracket defined by (\ref{2slabint}) and the conclusion follows from the integrability condition (\ref{tauhomgen}).\end{proof}
\begin{rem}\label{E*} {\rm We may transfer the previous Lie algebroid
structure of $\tilde{E}$ to $E^*$. Thus, Proposition \ref{integr1} may be
reformulated in terms of $E^*$. On the other hand, we may transfer the Lie
algebroid structure to any subbundle $\mathcal{Q}$  such that
$\mathbf{T}M=E^{\perp_g}\oplus\mathcal{Q}$; then, $\tau'$ yields an isomorphism
$\tau'_{\mathcal{Q}}:\mathcal{Q}\rightarrow E$.}\end{rem}

Furthermore, for any weakly integrable, generalized, 2-nilpotent structure $\tau$
with $2$-form $\omega$ and the corresponding inverse $\Lambda$, we have the
Gelfand-Dorfman dual bracket
\cite{D}
\begin{equation}\label{crosetGD} \{ \mathcal{X}_{mod\,E^{\perp_g}},
\mathcal{Y}_{mod\,E^{\perp_g}}\}_\Lambda=L_{\sharp_\Lambda
\mathcal{X}_{mod\,E^{\perp_g}}} \mathcal{Y}_{mod\,E^{\perp_g}}
-L_{\sharp_\Lambda\mathcal{Y}_{mod\,E^{\perp_g}}}
\mathcal{X}_{mod\,E^{\perp_g}}\end{equation} $$- d_{E}(\Lambda(
\mathcal{X}_{mod\,E^{\perp_g}}, \mathcal{Y}_{mod\,E^{\perp_g}})),$$ where
the Lie derivative and the differential are those of the Lie algebroid $E$. We
continue to use the identification of $E^*$ with $\mathbf{T}M/E^{\perp_g}$ by
the $g$-pairing and the identification of
$\sharp_\Lambda\mathcal{X}_{mod\,E^{\perp_g}}$ with $\tau\mathcal{X}$. Then, the
evaluation of the bracket (\ref{crosetGD}) on $\tau\mathcal{Z}\in\Gamma E$
yields
$$ \begin{array}{c}
<\{ \mathcal{X}_{mod\,E^{\perp_g}},
\mathcal{Y}_{mod\,E^{\perp_g}}\}_\Lambda,\tau\mathcal{Z}>\vspace*{2mm}\\
=-d_E\omega(\tau\mathcal{X},
\tau\mathcal{Y},\tau\mathcal{Z})+
\omega(\tau\mathcal{Z},[\tau\mathcal{X},\tau\mathcal{Y}]),
\end{array}$$
which is equivalent to
\begin{equation}\label{aux2}
\{ \mathcal{X}_{mod\,E^{\perp_g}}, \mathcal{Y}_{mod\,E^{\perp_g}}\}_\Lambda=
\tau^{'-1}[\sharp_\Lambda\mathcal{X}_{mod\,E^{\perp_g}},
\sharp_\Lambda\mathcal{Y}_{mod\,E^{\perp_g}}]\end{equation} $$-
i((\sharp_\Lambda\mathcal{X}_{mod\,E^{\perp_g}})
\wedge(\sharp_\Lambda\mathcal{Y}_{mod\,E^{\perp_g}}))d_E\omega.$$
\begin{prop}\label{intcuGD} The weakly integrable, generalized,
2-nilpotent structure $\tau$ is integrable iff
$(E^*,\{\,,\,\}_\Lambda,\sharp_\Lambda)$ is a Lie algebroid.\end{prop}
\begin{proof} If $\tau$ is integrable, $\omega_E$ is
symplectic, $\Lambda$ is a Poisson structure and
$(E^*,\{\,,\,\}_\Lambda,\sharp_\Lambda)$ is the corresponding, dual Lie
algebroid. Conversely, if $(E^*,\{\,,\,\}_\Lambda, \sharp_\Lambda)$ is a Lie
algebroid and we apply its anchor to (\ref{aux2}), we get
$$\sharp_\Lambda(i((\sharp_\Lambda\mathcal{X}_{mod\,E^{\perp_g}})
\wedge(\sharp_\Lambda\mathcal{Y}_{mod\,E^{\perp_g}}))d_E\omega)=0,$$
which is equivalent with $d_E\omega=0$.\end{proof}
\begin{rem}\label{obs3} {\rm Proposition \ref{intcuGD} is just
the known fact that $\Lambda$ is a Poisson bivector of $E$ iff
$\omega_E$ is a symplectic form. Essentially, the proposition tells that the
weakly integrable, generalized, 2-nilpotent structure $\tau$ is integrable iff
$(E,E^*)$ has a natural structure of a triangular Lie bialgebroid.}\end{rem}
\section{Metrics and Symplectic Dirac structures}
In this section we discuss symplectic big-isotropic and Dirac structures $E=im\tau$ on a generalized Riemannian manifold $(M,G)$ and we shall use again the notation of Sections 2 and 4. We start with the following remarks.
The endomorphism $\phi\in End( \mathbf{T})$ associated with $G$ is both a
$g$-isometry and a $G$-isometry. Firstly, this implies that $\phi(E)$ is again a $g$-isotropic subbundle of $\mathbf{T}M$ and $\phi(E^{\perp_g})=(\phi(E))^{\perp_g}$. Secondly, these properties imply the relations
$$(E^{\perp_G})^{\perp_g}=(E^{\perp_g})^{\perp_G}=\phi(E).$$
Thus, if the
subbundle $S$ is such that $\label{Eperpg}E^{\perp_g}=E\oplus_{\perp_G}S$, we
have a decomposition
\begin{equation}\label{descortog}
\mathbf{T}M=(E\oplus_{\perp_G}\phi(E))\oplus_{\perp_G}S,\end{equation}
where the subbundles $E\oplus_{\perp_G}\phi(E)$ and $S$ are invariant by $\phi$.

By (\ref{descortog}), since $E^{\perp_g}=ker\,\tau$, the mapping $\tau'_{\phi(E)}:\phi(E)\rightarrow E$ defined by
$\tau|_{\phi(E)}$ is an isomorphism.
\begin{defin}\label{deftauG} {\rm The structures $G$, $\tau$ are
{\it compatible}, and the pair $(G,\tau)$ is a generalized, metric, 2-nilpotent structure, if $\tau'_{\phi(E)}$ is a $G$-isometry, i.e.,
\begin{equation}\label{Gtau} G(\tau \mathcal{X},\tau \mathcal{Y})
=G(\mathcal{X},\mathcal{Y}), \hspace{5mm}\forall
\mathcal{X},\mathcal{Y}\in\phi(E). \end{equation}}\end{defin}

The following proposition gives several equivalent conditions.
\begin{prop}\label{equiGtau} The structures $G$ and $\tau$ are
compatible iff one of the following conditions holds:

1) for any $\mathcal{X},\mathcal{Y}\in T_x^{big}M$ $(x\in M)$ one has
\begin{equation}\label{Gtaunou} G(\tau\phi\tau\mathcal{X},
\tau\phi\tau\mathcal{Y})
=G(\tau\mathcal{X},\tau\mathcal{Y});
\end{equation}

2) the form $\omega_E$ associated with $\tau$ satisfies the condition
\begin{equation}\label{Gtaucuomega}
\omega_E(\lambda\tau\mathcal{X},\lambda\tau\mathcal{Y})=\omega_E(\tau\mathcal{X},
\tau\mathcal{Y}),\end{equation} where $\lambda=\tau\circ\phi:E\rightarrow
E$;

3) the morphism $\lambda=\tau\circ\phi:E\rightarrow E$ is a complex structure on
$E$ (i.e., $\lambda^2=-Id$);

4) the morphism  $\lambda'=\phi\circ\tau:\phi(E)\rightarrow
\phi(E)$ is a complex structure on $\phi(E)$ ($\lambda^{'2}=-Id$);

5) the morphism $\tilde{\lambda}=\tau\circ\phi:\mathbf{T}M\rightarrow
\mathbf{T}M$ satisfies the condition $\tilde\lambda^3+\tilde\lambda=0$;

6) the morphism  $\tilde\lambda'=\phi\circ\tau:\mathbf{T}M\rightarrow
\mathbf{T}M$ satisfies the condition
$\tilde\lambda^{'3}+\tilde\lambda'=0$.\end{prop}
\begin{proof} Condition 1) is equivalent to (\ref{Gtau}) because
the general expression of elements of $\phi(E)$ is
$\phi\tau\mathcal{X},\phi\tau\mathcal{Y}$ and $\phi$ is a $G$-isometry.

Furthermore, rewrite (\ref{Gtaunou}) as
$$g(\phi\lambda\tau\mathcal{X},\lambda\tau\mathcal{Y})
=g(\phi\tau\mathcal{X},\tau\mathcal{Y}).$$ The definition of $\omega_E$
transforms the latter into the equality
\begin{equation}\label{auxptlambda}
\omega_E(\lambda^2\tau\mathcal{X},\lambda\tau\mathcal{Y})
=\omega_E(\lambda\tau\mathcal{X},\tau\mathcal{Y}),\end{equation} which, therefore,
also is equivalent with the compatibility between $G$ and $\tau$. Since
$\lambda$ is an isomorphism of $E$, we may take
$\lambda\tau\mathcal{X}=\tau\mathcal{U}$ and we see that (\ref{auxptlambda}) is equivalent to (\ref{Gtaucuomega}).

On the other hand, using the definition of $\omega_E$ and the $g$-skew-symmetry of $\tau$, we get
$$
\omega_E(\lambda\tau\mathcal{X},\tau\mathcal{Y})=
-\omega_E(\tau\mathcal{X},\lambda\tau\mathcal{Y}).$$ This implies the
equivalence of (\ref{auxptlambda}) with
$$\omega_E(\lambda^2\tau\mathcal{X},\lambda\tau\mathcal{Y})
=-\omega_E(\tau\mathcal{X},\lambda\tau\mathcal{Y}),$$ which is
equivalent with $\lambda^2=-Id$ because $\omega_E$ is non degenerate.
Thus, we have proven conditions 2) and 3).

Then, $\lambda^2=-Id$ is equivalent to
$$g(\tau\phi\tau\phi\tau\mathcal{X},\phi\mathcal{Y})=-g(\tau\mathcal{X},
\phi\mathcal{Y}),\;\;\forall\mathcal{X},\mathcal{Y}\in\Gamma
\mathbf{T}M,$$ which transforms into
$$g(\tau\mathcal{X},\lambda^{'2}\phi\mathcal{Y})=-g(\tau\mathcal{X},\phi\mathcal{Y}).$$
It follows that $\lambda^2=-Id$ is equivalent with $\lambda^{'2}=-Id$, which is
condition 4).

Finally, since $\phi$ preserves the subbundle $S$, $\tilde\lambda$ vanishes on
$S\oplus\phi(E)$, which, together with $\lambda^2=-Id$ implies
$\tilde\lambda^3+\tilde\lambda=0$ and conversely. Similarly, $\lambda^{'2}=-Id$
is equivalent to $\tilde\lambda^{'3}+\tilde\lambda'=0$. This proves conditions 5) and 6).\end{proof}

Using (\ref{eqGg}) we see that
\begin{equation}\label{Hermit} G(\tau\mathcal{X},\tau\mathcal{Y})
=\omega_E(\tau\mathcal{X},\lambda\tau\mathcal{Y}).\end{equation} Then, if we use for
Lie algebroids the same terminology as for manifolds, we have
\begin{prop}\label{propHermit} If $(M,G)$ is a generalized
Riemannian manifold and $E$ is a $g$-isotropic subbundle of $\mathbf{T}M$, there
exist a bijection between the set of generalized, metric, 2-nilpotent structures
$\tau$ with $im\,\tau=E$ and the set of complex structures $\lambda$ on $E$ that
are compatible with $G|_E$. The structure $\tau$ is integrable iff $E$ is closed
by Courant brackets and $(G|_E,\lambda)$ is an almost K\"ahler structure on the
Lie algebroid $E$.
\end{prop}
\begin{proof} For a given $\lambda$ that satisfies the hypotheses, formula (\ref{Hermit}) yields $\omega_E$, which, then, produces
the following structure $\tau$:
$$ \tau|_{E^{\perp_g}}=0,\;\;\tau|_{\phi(E)}=\lambda\circ\phi.$$\end{proof}
\begin{defin}\label{tipK} {\rm An integrable, generalized, metric, 2-nilpotent
structure $(G,\tau)$ is of the {\it K\"ahler type} if the associated complex
structure $\lambda$ is integrable in the sense that it has a vanishing $E$-Nijenhuis tensor; the latter is defined like the usual Nijenhuis tensor but with brackets in $\Gamma E$.}\end{defin}

The Riemannian Lie
algebroid $(E,G|_E)$ has a Levi-Civita $E$-connection $D^E$
\cite{Bou} and we have
\begin{prop}\label{propKN} An integrable, generalized, metric, 2-nilpotent
structure $(G,\tau)$ is of the {\it K\"ahler type} iff
$D^E_\mathcal{X}\lambda=0$, $\forall\mathcal{X}\in\Gamma E$.
\end{prop}
\begin{proof}
The same calculations like in the proof of Proposition IX.4.2 of
\cite{KN} (with different factor conventions) give the formula
$$ \begin{array}{lcl}G((D^E_{\mathcal X}
\lambda)(\mathcal{Y}),\mathcal{Z}) &=&
\frac{1}{2}[d_E\omega_E(\mathcal{X},\mathcal{Y},\mathcal{Z})-
d_E\omega_E(\mathcal{X},\lambda\mathcal{Y},\lambda\mathcal{Z})\vspace{2mm}\\
&+&G(\mathcal{N}_\lambda(
\mathcal{X},\mathcal{Y}),\lambda\mathcal{Z}],\end{array}$$
for all $ \mathcal{X},\mathcal{Y},\mathcal{Z}\in\Gamma E$. The latter proves the required conclusion.
\end{proof}

The operators $ \tilde{\lambda},\tilde{\lambda}'$ are not generalized
F-structures
\cite{V-CRF} because they are not $g$-skew-symmetric. However, we
have
\begin{prop}\label{Fstr} A generalized, metric,
2-nilpotent structure $(G,\tau)$ has a canonically associated generalized,
metric F-structure.\end{prop}
\begin{proof} See  \cite{V-CRF} for the definition of generalized, metric
F-structures. The required structure is defined by
\begin{equation}\label{Phi2}\Phi=\tilde\lambda+\tilde\lambda'=
\tau\phi+\phi\tau.\end{equation} The properties of
$\tilde\lambda,\tilde\lambda'$ proven in Proposition
\ref{equiGtau} imply $\Phi^3+\Phi=0$. The metric compatibility conditions
$$g(\Phi\mathcal{X},\mathcal{Y})+
g(\mathcal{X},\Phi\mathcal{Y})=0,\; G(\Phi\mathcal{X},\mathcal{Y})+
G(\mathcal{X},\Phi\mathcal{Y})=0$$ easily check
for all possible combinations of arguments in $E,\phi(E),S$. (We
have to use the facts that $\phi$ is a $g$-isometry and that $S\perp_g
E,S\perp_G E$.)\end{proof}

Let us restrict ourselves to the almost tangent case. Then,
$ \mathbf{T}M=E\oplus\phi(E)$ and the structure
$(G,\Phi)$ associated to $(G,\tau)$ is a generalized almost Hermitian structure ($\Phi^2=-Id$). On the other hand, the pair $(G,\tau)$ has the associated,
generalized, almost paracomplex structure $\Psi=\Psi_E$ defined in Section 3,
i.e., $$ \Psi|_E=Id,\;\;
\Psi|_{\phi(E)}=-Id,$$  which is $G$-compatible. Furthermore, by checking separately on $E,\phi(E)$, we get
$$\Phi\circ\Psi=\Psi\circ\Phi=\tau\phi-\phi\tau=
\tilde{\lambda}-\tilde{\lambda}'.$$
Together with the expression (\ref{Phi2}) of $\Phi$ this leads to the equality
\begin{equation}\label{taucuPsi}
\tau=\frac{1}{2}\Phi\circ(Id+\Psi)\circ\phi.\end{equation}
\begin{prop}\label{thtaucuPsi} On a generalized Riemannian
manifold $(M,G)$, there exists a canonical bijection between the $G$-compatible,
generalized, almost tangent structures $\tau$ and the set of commuting pairs
$(\Phi,\Psi)$ where $\Phi$ is a $G$-compatible, generalized, almost complex
structure and $\Psi$ is a $G$-compatible, generalized, almost paracomplex
structure on $M$.\end{prop} \begin{proof} We have seen how to construct the pair
$(\Phi,\Psi)$ from $\tau$. Conversely, for a given pair $(\Phi,\Psi)$, let us
define $\tau\in End(\mathbf{T}M)$ by formula (\ref{taucuPsi}). Since $\phi$ and
$\Phi$ are isomorphisms, we see that $im\,\tau=im(Id+\Psi)$, which is the
$(+1)$-eigenbundle of $\Psi$ and has rank $m$. Thus, we will define this
subbundle as $E$ and, necessarily, the $(-1)$-eigenbundle of $\Psi$ will be
$\phi(E)$. It is easy to check that $\tau^2=0$ on both $E$ and $\phi(E)$. For this structure $\tau$, we have
$$\lambda=\tau\circ\phi=\frac{1}{2}\Phi\circ(Id+\Psi)|_E$$ and
the commutation between $\Phi$ and $\Psi$ yields $\lambda^2=-Id$. Thus, by 3) of Proposition \ref{equiGtau} $\tau$ is $G$-compatible.\end{proof}
\begin{prop}\label{EPhi} On a generalized Riemannian
manifold $(M,G)$, there exists a canonical bijection between the $G$-compatible,
generalized, almost tangent structures $\tau$ and the set of pairs $(E,\Phi)$
where $E$ is an almost Dirac structure and $\Phi$ is a $G$-compatible,
generalized, almost complex structure such that $\Phi(E)\subseteq E$.
Furthermore, the structure $\tau$ is integrable iff $E$ is integrable and $\Phi$
satisfies the following condition
\begin{equation}\label{integrcuPhi} [\Phi\mathcal{X},\phi\mathcal{Y}]+
[\phi\mathcal{X},\Phi\mathcal{Y}]+ \phi\Phi[\Phi\mathcal{X},\Phi\mathcal{Y}]\in
\Gamma E,\;\;\forall\mathcal{X},\mathcal{Y}\in\Gamma E.\end{equation}
\end{prop}
\begin{proof} The structure $\Phi$ associated with $\tau$ is (\ref{Phi2}) again. In the converse direction, take $\Psi=\Psi_E$ in Proposition \ref{thtaucuPsi}, alternatively, take
$$ \tau|_E=0,\;\tau|_{\phi(E)}=(\Phi\circ\phi)|_E.
$$ Furthermore, notice that if $E$ is integrable,
then, $ \mathcal{N}_\tau$ vanishes if at least one of the arguments is in
$\Gamma E$. Furthermore, since $\Phi|_E=\tau\circ\phi$, the remaining part of
the $\tau$-integrability condition reduces to
$$ [\Phi\mathcal{X},\Phi\mathcal{Y}]
-\tau([\Phi\mathcal{X},\phi\mathcal{Y}]+
[\phi\mathcal{X},\Phi\mathcal{Y}])=0,\;\;
\forall\mathcal{X},\mathcal{Y}\in\Gamma E.
$$
Since $E$ is $\Phi$-invariant, we may replace
$$[\Phi\mathcal{X},\Phi\mathcal{Y}]= -\Phi^2[\Phi\mathcal{X},\Phi\mathcal{Y}]= -\tau\phi\Phi[\Phi\mathcal{X},\Phi\mathcal{Y}]$$
and we get the integrability condition (\ref{integrcuPhi}).
\end{proof}

Now, in analogy to Section 3, we prove
\begin{prop}\label{taunabla} For any generalized 2-nilpotent structure $\tau$,
there exist big connections $\nabla$ that commute with $\tau$. \end{prop}
\begin{proof} First, we notice that a big connection $\nabla$ commutes with
$\tau$ iff $\nabla$ preserves the subbundle $E=im\,\tau$ and the induced connection
$\nabla'$ of $E$ preserves the corresponding $2$-form $\omega_E$. The
preservation of $E$ obviously is a necessary condition for
$\nabla\tau=\tau\nabla$. Thus, $\nabla'$ exists and the
definition of $\omega_E$ shows that $$\nabla'_X\omega_E(e,\tau\mathcal{Y})=
\nabla_Xg(e,\mathcal{Y})=0,\;\;e\in\Gamma E,\mathcal{Y}\in\Gamma \mathbf{T}M.$$
Conversely, by subtracting $\nabla_Xg(e,\mathcal{Y})=0$ from $\nabla'_X\omega_E(e,\tau\mathcal{Y})=0$
we get $$g(e,\nabla_X\mathcal{Y})=
\omega_E(e,\tau\nabla_X\mathcal{Y})=\omega_E(e,\nabla_X\tau\mathcal{Y}),$$
therefore, $\nabla\tau\mathcal{Y}=\tau\nabla\mathcal{Y}$.

Now, in order to get the required big connection $\nabla$ we first construct an
$\omega_E$-preserving connection $\nabla'$ on $E$. $\nabla'$ is given by the known
formulas of almost symplectic geometry (e.g., \cite{Vcs}), for instance
$$\nabla'_Xe=\nabla^0_Xe+\Theta(X,e),\;\; \omega_E(\Theta(X,e),e')=\frac{1}{2}
\nabla^0_X\omega_E(e,e'),$$ where $\nabla^0$ is an arbitrary
connection on the vector bundle $E$ and $\Theta:\Gamma TM\times\Gamma E$ is a
``tensor". Then, we take a metric connection $\nabla^S$ on the pseudo-Euclidean
subbundle $(S,g|_S)$ of (\ref{descortog}). Finally, we define the connection
$\nabla''$ on $\phi(E)$ such that $\nabla'+\nabla''$ preserves $g|_{E\oplus\phi(E)}$ (in the identification of $\phi(E)$ with $E^*$, $\nabla''$ is $\nabla'$ acting on $E^*$). With these choices, $\nabla=\nabla'+\nabla^S+\nabla''$ is a big
connection that preserves $E$ and induces the $\omega_E$-preserving connection $\nabla'$ on $E$,
hence, $\nabla$ commutes with $\tau$.
\end{proof}

Furthermore, in analogy with Proposition \ref{integrcuconex}, we have
\begin{prop}\label{intorstg} If $\nabla$ is a big connection that
commutes with the integrable, generalized 2-nilpotent structure $\tau$, the
Gualtieri torsion of $\nabla$ satisfies the conditions
\begin{equation}\label{ttauintegr1} \mathcal{T}^\nabla( \mathcal{X},\mathcal{Y},\mathcal{Z})=0,
\end{equation}
if $\mathcal{X},\mathcal{Y},\mathcal{Z}\in\Gamma E^{\perp_g}$,
\begin{equation}\label{ttauintegr2} \mathcal{T}^\nabla( \tau\mathcal{X},\tau\mathcal{Y},\mathcal{Z})+
\mathcal{T}^\nabla( \tau\mathcal{X},\mathcal{Y},\tau\mathcal{Z})
\mathcal{T}^\nabla( \mathcal{X},\tau\mathcal{Y},\tau\mathcal{Z}),
\end{equation} if none of the arguments
$\mathcal{X},\mathcal{Y},\mathcal{Z}$ belongs to $E^{\perp_g}$. Conversely, if
there exists a big connection that commutes with $\tau$ and satisfies
(\ref{ttauintegr1}), (\ref{ttauintegr2}), $\tau$ is integrable.\end{prop}
\begin{proof} Formula (\ref{NijcutG}) with $\epsilon=0$ shows that the torsion condition
$$\mathcal{T}^\nabla( \tau\mathcal{X},\tau\mathcal{Y},\mathcal{Z})+
\mathcal{T}^\nabla( \tau\mathcal{X},\mathcal{Y},\tau\mathcal{Z})
+\mathcal{T}^\nabla( \mathcal{X},\tau\mathcal{Y},\tau\mathcal{Z})=0$$ for
arbitrary arguments makes the assertions of the proposition hold. This torsion condition is equivalent to the couple
(\ref{ttauintegr1}), (\ref{ttauintegr2}) (recall that $E^{\perp_g}=ker\tau$).
\end{proof}
\begin{prop}\label{lambdators} Let $\tau$ be an integrable $G$-compatible 2-nilpotent structure on the generalized Riemannian manifold $(M,G)$. If $\nabla$ is a $G$-metric big connection that commutes with $\tau$, then, the $E$-connection induced by $\nabla$ on $E$ is the $E$-Levi-Civita connection and the structure $\tau$ is of the K\"ahler type.\end{prop}
\begin{proof} Let $\nabla'$ be the usual connection induced by $\nabla$ on $E$, which is known to preserve the $2$-form $\omega_E$. Under the hypotheses of the corollary, it also preserves the metric $G|_E$ and the complex structure $\lambda$. The integrability condition (\ref{ttauintegr1}) implies $T^\nabla(e_1,e_2)=0$. The induced $E$-connection is defined by $\nabla^E_{e_1}e_2=\nabla'_{pr_{TM}e_1}e_2$ and, by the previous remarks, it follows that $\nabla^E$ is torsionless and preserves $G|_E$ and $\lambda$, which means that $\nabla^E$ is the $E$-Levi-Civita connection and $\tau$ is of the K\"ahler type.\end{proof}


\begin{thebibliography}{xx}
\bibitem{Bou} M. Boucetta, Riemannian Geometry of Lie Algebroids,
arXiv:0806.3522v2[mathDG].
\bibitem{C} T. Courant, Dirac Manifolds, Transactions Amer. Math.
Soc., 319 (1990), 631-661.
\bibitem{D} I. Dorfman, Dirac structures and integrability of nonlinear evolution
equations, Nonlinear Science: Theory and Applications. John Wiley
\& Sons, Ltd., Chichester, 1993.
\bibitem{E} I. T. Ellwood, NS-NS fluxes in Hitchin's generalized
geometry, arXiv:hep-th/0612100v4.
\bibitem{Galt} M. Gualtieri, Generalized complex geometry, Ph.D.
thesis, Univ. Oxford, 2003; arXiv:math.DG/0401221.
\bibitem{G2007} M. Gualtieri, Branes of Poisson variety, in
: The many Facets of Geometry. A tribute to Nigel Hitchin (O. Garcia-Prado, J.
P. Bourguignon and S. salamon, eds.), Oxford Univ. Press, Oxford, 2010, 368-395.
\bibitem{Ht1} N. J. Hitchin, Generalized Calabi-Yau manifolds,
Quart. J. Math., 54 (2003), 281-308.
\bibitem{KN} S. Kobayashi and K. Nomizu, Foundations of Differential
Geometry, vol. II, Interscience Publ., New York, 1969.
\bibitem{LWX} Z.-J. Liu, A. Weinstein and P. Xu, Manin triples for
Lie bialgebroids, J. Diff. Geom., 45 (1997), 547-574.
\bibitem{MX} K. Mackenzie and P. Xu, Lie bialgebroids and Poisson
groupoids, Duke Math. J., 73 (1994), 415-452.
\bibitem{PW} Y. S. Poon and A. wade, Generalized contact structures, J.
London Math. Soc., 83 (2011), 309-332.
\bibitem{SW} P. \v{S}evera and A. Weinstein, Poisson geometry
with a $3$-form background, Prog. Theor. Phys. Suppl. 144 (2001), 145-154.
\bibitem{Vcs} Symplectic curvature tensors, Monatshefte f\"ur Math.,
100 (1985) 299-327.
\bibitem{V-carte} I. Vaisman, Lectures on the geometry of Poisson
manifolds, Progress in Math., vol. 118, Birkh\"auser Verlag, Boston, 1994.
\bibitem{IV} I. Vaisman, Reduction and submanifolds of generalized
complex manifolds, Diff. Geom. Appl., 25 (2007), 147-166.
\bibitem{V2} I. Vaisman, Isotropic Subbundles of $TM\oplus T^*M$,
Intern. J. of Geom. Methods in Modern Physics, 4(3) (2007),487-516.
\bibitem{V-CRF} I. Vaisman, Generalized CRF-structures, Geometriae
Dedicata, 133 (2008),129-154.
\bibitem{VgenS} I. Vaisman, From Generalized K\"ahler to
Generalized Sasakian Strucutres, J. of Geom. and Symmetry in Physics,
18 (2010), 63-86.
\bibitem{Wade} A. Wade, Dirac structures and paracomplex manifolds,
C. R. Acad. Sci. Paris, Ser. I, 338 (2004), 889-894.
\bibitem{Zab} M. Zabzine,
Lectures on Generalized Complex Geometry and Supersymmetry, Archivum
Mathematicum (Supplement), 42 (2006), 119-146.
\end{thebibliography}
\end{document}